\documentclass[a4paper,12pt,reqno]{amsart}
\usepackage[utf8]{inputenc}
\usepackage{csquotes}
\usepackage{amssymb}
\usepackage{amsmath}
\usepackage{makecell}
\usepackage{multirow}
\usepackage{hhline}
\usepackage{amsthm}
\usepackage{amsfonts}
\usepackage{mathtools}
\usepackage{todonotes}
\usepackage{comment}
\usepackage{mathrsfs}
\usepackage{enumitem}
\usepackage[all,cmtip]{xy}
\usepackage{tikz-cd}
\usepackage[hyphens]{url}
\usepackage{hyperref}
\usepackage{cleveref}
\usepackage{longtable}

\usepackage[british]{babel}
\usepackage[margin=6em]{geometry}
\usepackage{eqparbox}

\usepackage{pbox}

\newcommand{\set}[1]{\left\lbrace #1 \right\rbrace}

\newcommand{\field}[1]{\mathbb{#1}}  
\newcommand{\Q}{\field{Q}} 
\newcommand{\N}{\field{N}} 
\newcommand{\C}{\field{C}} 
\newcommand{\Z}{\field{Z}} 
\newcommand{\F}{\field{F}} 

\renewcommand{\P}{\field{P}}
\newcommand{\PP}{\field{P}}

\newcommand{\floor}[1]{\left\lfloor #1 \right\rfloor}

\newcommand{\petar}[1]{{\color{orange} ($\clubsuit$ Petar: #1)}}
\newcommand{\maarten}[1]{{\color{blue} ($\spadesuit$ Maarten: #1)}}

\DeclareMathOperator{\End}{End}

\DeclareMathOperator{\SL}{SL}

\DeclareMathOperator{\Pic}{Pic}

\DeclareMathOperator{\Sym}{Sym}

\DeclareMathOperator{\genus}{genus}

\newcommand{\lmfdbec}[3]{\href{https://www.lmfdb.org/EllipticCurve/Q/#1/#2/#3}{#1.#2#3}}
\newcommand{\lmfdbmf}[4]{\href{https://www.lmfdb.org/ModularForm/GL2/Q/holomorphic/#1/#2/#3/#4}{#1.#2.#3.#4}}

\newcommand{\computationbound}{3000}

\newtheorem{lemma}{Lemma}
\newtheorem{theorem}[lemma]{Theorem}
\newtheorem{proposition}[lemma]{Proposition}
\newtheorem{corollary}[lemma]{Corollary}

\theoremstyle{definition}
\newtheorem{definition}[lemma]{Definition}

\newtheorem{conjecture}[lemma]{Conjecture}
\newtheorem{remark}[lemma]{Remark}
\newtheorem{setup}[lemma]{Setup}

\numberwithin{lemma}{section}
\numberwithin{equation}{section} 
\numberwithin{figure}{section}

\title{Morphisms on the modular curve $X_0(p)$ and degree $6$ points}

\author{\sc Maarten Derickx}
\address{Maarten Derickx\\
University of Zagreb\\  
Bijeni\v{c}ka Cesta 30 \\
10000 Zagreb\\
Croatia}
\email{maarten@mderickx.nl}
\urladdr{http://www.maartenderickx.nl/}

\author{\sc Petar Orli\'c}
\address{Petar Orli\'c \\
University of Zagreb\\  
Bijeni\v{c}ka Cesta 30 \\
10000 Zagreb\\
Croatia}
\email{petar.orlic@math.hr}

\begin{document}
\begin{abstract}
    Let $p$ be a prime. We study non-constant morphisms $f:X_0(p)_\Q \to Y$, where $Y/\Q$ is a curve of genus $\geq 2$. We prove that for $p<\computationbound$ such an $f$ of degree $d>1$ must be isomorphic to the quotient map $X_0(p)\to X_0^+(p)$. Supported by computational and theoretical evidence, we also conjecture that this is true for all primes $p$.
    
    These results allow us to classify all points of degree $\leq 25$ on $X_0(p)$ that come from a map to some curve of genus $\geq 2$. As an application, we were able to determine all curves $X_0(p)$ with infinitely many points of degree $6$ over $\Q$ except for $p=193$, continuing the previous results on small degree points on $X_0(N)$.
\end{abstract}


\maketitle

\section{Introduction}

Modular curves $X(\Gamma)$ are a type of algebraic curves which can be constructed as quotients of the compactified upper half plane $\mathcal{H}^*$ with $\Gamma$, a congruence subgroup of the modular group $\textup{SL}_2(\Z)$. Examples of modular curves are $X_0(N)$ and $X_1(N)$ which correspond to congruence subgroups

\begin{align*}
\Gamma_0(N)&=\left\{ 
\begin{bmatrix}
a & b\\
c & d
\end{bmatrix}
\in \SL_2(\Z) : c\equiv 0\pmod{N} \right\},\\
\Gamma_1(N)&=\left\{ 
\begin{bmatrix}
a & b\\
c & d
\end{bmatrix}
\in \SL_2(\Z) : a,d\equiv 1\pmod{N}, \ c\equiv 0\pmod{N}
\right\}.
\end{align*}

Morphisms $X_0(N)\to Y$ defined over number fields are an important topic in arithmetic geometry and have been the subject of extensive research. Let us first discuss the case when $g(Y)=0$ and $Y$ is isomorphic to $\PP^1$.

Let $C$ be a smooth projective curve over a field $k$. The $k$-gonality of $C$, denoted by $\textup{gon}_k C$, is the least degree of a non-constant $k$-rational morphism $f:C\to\mathbb{P}^1$. Thus, determining the gonality of $X_0(N)$ is very important if one wants to study morphisms $X_0(N)\to\PP^1$.

The gonality of modular curves has been extensively studied.
For example, there are upper bounds on the gonality of any algebraic curve in terms of its genus $g$, see \cite[Proposition A.1]{Poonen2007}. Moreover, there exists a lower bound for the $\C$-gonality for any modular curve $X(\Gamma)$, linear in terms of the index of the corresponding congruence subgroup $\Gamma$ \cite{abramovich}, \cite[Appendix 2]{ Kim2002}.

Regarding the curve $X_0(N)$, all cases when its $\C$-gonality is at most $4$ and $\Q$-gonality is at most $6$ have been determined \cite{Ogg74, Bars99, HasegawaShimura_trig, JeonPark05, NajmanOrlic22}.

Let us now consider the case when $g(Y)=1$ and $Y=E$ is an elliptic curve. Determining all possible degrees of morphisms to $E$ is useful when searching for small degree points because the only curves with infinitely many rational points are those of genus $0$ isomorphic to $\PP^1$ and positive rank elliptic curves. If one wants to determine all possible degrees of rational morphisms $X_0(N)\to E$, there are two cases with respect to the conductor of $E$: $\textup{cond}(E)=N$ and $\textup{cond}(E)<N$ (\textup{cond}(E) is a proper divisor of $N$). The case $\textup{Cond}(E)=N$ is easier and can be dealt with using the Modularity theorem.

\begin{theorem}[Modularity theorem]\label{modularity_thm}
    Let $E/\Q$ be an elliptic curve of conductor $N$. Then there exists a non-constant rational morphism $f:X_0(N)\to E$. Furthermore, if $f$ has the least possible degree, then any other morphism $g:X_0(N)\to E$ factors through $f$.
\end{theorem}

Therefore, all rational morphisms $X_0(N)\to E$ are of the form $[d]\circ f$ for some $d\in\Z$. This holds even for CM curves because the extra endomorphisms are not defined over $\Q$. The case $\textup{cond}(E)<N$ is more involved. Derickx and Orlić \cite[Propositions 1.10, 1.11]{DerickxOrlic23} gave a method to compute all possible degrees of rational morphisms $X_0(N)\to E$ if $E$ has positive rank.

This paper focuses on the case when $g(Y)\geq2$. A useful result here is the De Franchis-Severi theorem.

\begin{theorem}[De Franchis-Severi]
    Let $X$ be a compact Riemann surface of genus $g\geq2$. Then the following two statements hold:
    \begin{enumerate}[label=(\alph*)]
        \item For a fixed compact Riemann surface $Y$ of genus $\geq2$, the number of non-constant holomorphic maps $\pi: X\to Y$ is finite.
        \item There are only finitely many compact Riemann surfaces $Y$ of genus $\geq2$ which admit a non-constant holomorphic map from $X$.
    \end{enumerate}
    The sets in part $(a)$ and $(b)$ are not only finite, but are in fact explicitly computable \cite[Theorem 5.7]{BakerGonzalezPoonen2005}.
\end{theorem}

This differs from the genus $0$ or $1$ case where there can be infinitely many morphisms to $\PP^1$ or an elliptic curve. Therefore, it makes sense to try to determine possible morphisms from a fixed curve $X$ to a curve of genus $\geq2$, especially if $X$ is a modular curve.

Baker, Gonz{\'a}lez-Jim{\'e}nez, Gonz{\'a}lez and Poonen \cite{BakerGonzalezPoonen2005} consider the maps from $X_0(N)$ and $X_1(N)$ to curves of genus $g\geq2$ and get some interesting results. For example, they proved that for each $g\geq2$, the set of new modular curves (curves whose Jacobian is a quotient of $J_1(N)_{\textup{new}}$) over $\Q$ of genus $g$ is finite and computable. They also proved that for each $G\geq2$, the set of new modular curves over $\Q$ of gonality $\leq G$ is finite and computable and obtained several results giving restrictions on new modular hyperelliptic curves over $\Q$. We use some of their results in the proofs of our main theorems.

It is natural to try extending their results by considering the maps to general curves of genus $g\geq2$.

In this paper we study the modular curve $X_0(p)$ for $p$ prime. The reason for considering only the prime levels is that the arguments become simpler and more computable. For example, all simple factors $A\subseteq J_0(p)$ occur with multiplicity $1$ and are associated with newforms. In particular, $J_0(p)=J_0(p)^{new}$, enabling us to use the results of \cite{BakerGonzalezPoonen2005}. Moreover, $J_0(p)$ has only finitely many abelian subvarieties over $\Q$. This is not the case for composite levels, see \Cref{J0p_finite_subvarieties} and \Cref{J0N_infinite_subvarieties}.


Our first main result is the following theorem

\begin{theorem}\label{thm_p<3000}
Let $p<\computationbound$ be a prime and let $Y$ be a nice (smooth, projective, and geometrically integral) curve over $\Q$ of genus $\geq 2$. Suppose that $f: X_0(p) \to Y$ is a non-constant morphism of degree $d>1$ defined over $\Q$. Then $f$ is isomorphic to the quotient map $X_0(p) \to X_0^+(p)=X_0(p)/w_p$.
\end{theorem}

This leads to the question whether the results of \cref{thm_p<3000} are a consequence of the law of small numbers, or part of a more general phenomenon.

Recall that $J_0(p)^+:=(1+w_p)J_0(p)$ and $J_0(p)^-:=(1-w_p)J_0(p)$ (so that $J_0(p)\sim J_0(p)^+\oplus J_0(p)^-$). These are the largest abelian subvarieties of $J_0(p)$ on which the Atkin-Lehner involution $w_p$ as multiplication by $1$ and $-1$ respectively.

Our second main result shows that this is indeed try under certain conditions on the isogeny decompostion of $J_0(p)$.

\begin{theorem}\label{thm_jacobian_decomposition}
Let $p$ be a prime and let $J_0(p)$ be the Jacobian of the curve $X_0(p)$.  If $J_0(p)$ is isogenous to \[A^+ \oplus A^- \oplus A\] with $A^+ \subseteq J_0(p)^+$, $A^- \subseteq J_0(p)^-$ being simple abelian varieties and $A\subseteq J_0(p)$ such that $\dim A \leq 2$, then, up to isomorphism, the quotient map $X_0(p) \to X_0^+(p)$ is the only non-constant rational morphism of degree $d>1$ from $X_0(p)$ to a curve of genus $\geq 2$. 
\end{theorem}

Numerically, it seems that \Cref{thm_jacobian_decomposition} is applicable for a majority of prime levels $p$. Indeed, by analyzing the data computed in \cite{Best2021} (also available on LMFDB \cite{lmfdb}) and \cite[Section 9]{Cowan2022} it follows that it is applicable for 77636 out of the 78498 primes below 1,000,000 ($\approx 98.9\%$). Furthermore, we also computed this percentage for all primes below $B$ for several $B<1,000,000$ and this seems to suggest that this percentage is slowly climbing to $100\%$. 

Motivated by \Cref{thm_p<3000} and \Cref{thm_jacobian_decomposition}, we state the following conjecture.

\begin{conjecture}\label{x0n_morphism_conjecture}
    Let $p$ be a prime, $Y$ a curve of genus at least $2$, and $f: X_0(p) \to Y$ a non-constant morphism of degree at least $2$.  Then $f$ is isomorphic to the quotient map $X_0(p) \to X_0^+(p)$.
\end{conjecture}

There are two reasons we think why \Cref{x0n_morphism_conjecture} is true. First, the number of $A \subseteq J_0(p)$ not ruled out by \Cref{cor:not_curve_criterion} becomes more sparse with larger values of $p$. Second, the computational results of \cite{Cowan2022} show that it is very rare for curves $X_0(p)$ to have Jacobian factors of dimension $\displaystyle\leq\frac{2g(X_0(p))-2}{3}+1$ (factors of larger dimension are much easier to deal with: Riemann-Hurwitz formula implies a degree $2$ map from $X_0(p)$ which necessarily comes from an involution, see \Cref{deg2map_remark}). Moreover, if such factors occur, they seem to have a very small dimension and \cite[Theorem 1.3]{BakerGonzalezPoonen2005} shows that there is a finite and computable set of counterexamples for every genus $g(Y)$. 

In \cite[Section 6.5]{BakerGonzalezPoonen2005} it is proved that any counterexample to \Cref{x0n_morphism_conjecture} with $g(Y)=2$ has to be of level $\leq3000$. Together with \Cref{thm_p<3000} this gives the following corollary.
\begin{corollary}
    There are no counterexamples to \Cref{x0n_morphism_conjecture} with $g(Y)=2$.
\end{corollary}

These results can be used to determine the infiniteness of degree $d$ points on $X_0(p)$ for small $d$, which is an important property of an algebraic curve. Faltings' theorem tells us that a curve $C/\Q$ has infinitely many rational points if and only if $C$ is isomorphic to $\mathbb{P}^1$ ($g=0$) or $C$ is an elliptic curve ($g=1$) with positive rank. The next theorem solves the case $d=2$.

\begin{theorem}[{Harris, Silverman: \cite[Corollary 3]{HarrisSilverman91}}]\label{harrissilverman}
    Let $K$ be a number field, and let $C/K$ be a curve of genus at least $2$. Assume that $C$ is neither hyperelliptic nor bielliptic. Then the set of quadratic points on $C/K$ is finite.
\end{theorem}

Kadets and Vogt \cite[Theorem 1.2 and Proposition 6.3]{KadetsVogt} also gave a characterization of curves with infinitely many degree $d$ points for $d\leq5$, although their characterization relies on Debarre-Fahlaoui curves which can be difficult to apply. The reason for that is that we do not know how to prove in the general case that a curve is not Debarre-Fahlaoui. The only result in that direction we know is that Debarre-Fahlaoui curves must admit a non-constant morphism to a positive rank elliptic curve \cite[Proposition 4.6]{DHJOdensitydegree}. This will also be discussed later in the paper.

The question of determining whether the curve $C/K$ has infinitely many degree $d$ points over $K$ is closely related to determining the $K$-gonality of $C$, as can be seen from the following theorem.

\begin{theorem}[Frey \cite{frey}]\label{frey}
    Let $C$ be a curve defined over a number field $K$. If $C$ has infinitely many points of degree $\leq d$ over $K$, then $\textup{gon}_K C\leq2d$.
\end{theorem}

Using \Cref{thm_p<3000} we are able to obtain of final main result. Namely, a classification of all modular curves $X_0(p)$ of prime level $p$ with infinitely many points of degree $6$, with the exception for the level $p=193$ which remains open. See \Cref{known_results_section} for previous results on points for degree $\leq5$. 

\begin{theorem}\label{deg6_thm}
    Let $p\neq193$ be a prime. The modular curve $X_0(p)$ has infinitely many points of degree $6$ over $\Q$ if and only if $p\leq151$ or \[p\in\{163,167,179,181,191,227,239,269\}.\]
\end{theorem}

One of the key ways that \Cref{thm_p<3000} helps in the above proof is that the results of \cite{KadetsVogt} allow one to relate the study of degree $6$ points on $X_0(p)$ to maps on $X_0(p)$ to other curves whenever the genus of $X_0(p)$ is $\geq17$. See \Cref{KVcor} for a more precise statement tailored to degree $6$. In fact, the bound of $3000$ in \Cref{thm_p<3000} is enough to classify the points on $X_0(p)$ that come from maps to curves of genus $\geq 2$ for all degrees up to $25$. This classification is formalized in \Cref{thm:maps_for_small_degree_points}. 

The paper is organized as follows:
\begin{itemize}
    \item In \Cref{known_results_section} we present the known results on the curves $X_0(N)$ with infinitely many points of degree $d\leq5$.
    \item In \Cref{morphism_section} we determine all possible rational morphisms $f:X_0(p)\to Y$ for $p<\computationbound$, where $Y$ is a curve of genus $\geq2$, and give results on lower bounds for $\deg f$ for larger $p$.
    \item In \Cref{inf_deg6points_section} we prove that the curve $X_0(p)$ has infinitely many degree $6$ points for all levels $p$ listed in \Cref{deg6_thm}.
    \item In \Cref{finite_deg6points_section} we use the results from \Cref{morphism_section} to prove that all other curves $X_0(p)$ have only finitely many degree $6$ points over $\Q$. At the end of this section we also prove \Cref{deg6_thm} and give a short discussion about the points of degree $d\geq7$ on $X_0(p)$ and the level $p=193$.
\end{itemize}

One might ask whether it is possible to go further and try to determine all curves $X_0(p)$ with infinitely many points of degree $d$ for some $d\geq7$. However, a major obstacle for $d=7$ is that we were unable to determine all curves $X_0(p)$ with $\Q$-gonality equal to $7$. The smallest such prime is $p=163$ which has $\Q$-gonality either $7$ or $8$, see \cite[Table 2]{NajmanOrlic22}. We could not find a degree $7$ rational morphism to $\PP^1$ or disprove its existence. 

Moreover, it becomes much harder to determine all curves $X_0(N)$ with $\Q$-gonality equal to $d$ as $d$ grows larger. As the genus of $X_0(N)$ becomes larger, working with models and Riemann-Roch spaces over $\Q$ and $\F_p$ becomes much more difficult, making the computations increasingly unfeasible.

The results in this paper are based on \texttt{Magma} \cite{magma} and Sage computations. The computer code accompanying this paper can be found on \begin{center}
    \url{https://github.com/nt-lib/maps_on_X0_p}.
\end{center}

\section*{Acknowledgments}

We thank Alex Galarraga for pointing out a mistake regarding the level $p=193$ in the first version of the preprint. We also thank Filip Najman for his helpful comments.

The authors were supported by the Croatian Science Foundation under the project no. IP-2022-10-5008 and by the project “Implementation of cutting-edge research and its application as part of the Scientific Center of Excellence for Quantum and Complex Systems, and Representations of Lie Algebras“, PK.1.1.10.0004, European Union, European Regional Development Fund.

\section{Previous results on low degree points on modular curves.} \label{known_results_section}

Low degree points over $\Q$ on modular curves have been extensively studied. One of the reasons for this is that non-cuspidal points on modular curves represent isomorphism classes of elliptic curves. For example, rational points on curves $X_1(M,N)$ represent $\Q$-isomorphism classes of elliptic curves with torsion group $\Z/M\Z\times\Z/N\Z$, and rational points on curves $X_0(N)$ represent 
$\overline{\Q}$-isomorphism classes of elliptic curves together with a cyclic $N$-isogeny.

We present the results for the curve $X_0(N)$ as this paper is focused on it. All curves $X_0(N)$ with infinitely many degree $d$ points have been determined for $d\leq4$.

\begin{theorem}[{Bars \cite{Bars99}}]\label{thmquadratic}
    The modular curve $X_0(N)$ has infinitely many points of degree $2$ over $\Q$ if and only if
    $$N\in\{1-33,35-37,39-41,43,46-50,53,59,61,65,71,79,83,89,101,131\}.$$
\end{theorem} 

\begin{theorem}[{Jeon \cite{Jeon2021}}]\label{thmcubic}
    The modular curve $X_0(N)$ has infinitely many points of degree $3$ over $\Q$ if and only if
    $$N\in\{1-29,31,32,34,36,37,43,45,49,50,54,64,81\}.$$
\end{theorem}

\begin{theorem}[{Hwang, Jeon; Derickx, Orlić \cite{Hwang2023, DerickxOrlic23}}]\label{quarticthm}
    The modular curve $X_0(N)$ has infinitely many points of degree $4$ over $\Q$ if and only if
    \begin{align*}
        N\in\{&1-75,77-83,85-89,91,92,94-96,98-101,103,104,107,111,\\
        &118,119,121,123,125,128,131,141-143,145,155,159,167,191\}.
    \end{align*}
\end{theorem}

The problem of determining all curves $X_0(N)$ with infinitely many degree $5$ points has been partially solved and only finitely many levels $N$ remain unsolved. For all unsolved levels the curve $X_0(N)$ has infinitely many quartic points, making the arguments more complicated.

\begin{theorem}[{Derickx, Hwang, Jeon, Orlić \cite{DHJOdensitydegree}}]\label{quinticthm}
    The modular curve $X_0(N)$ has infinitely many points of degree $5$ and only finitely many points of degree $\leq4$ over $\Q$ if and only if $N=109$.
\end{theorem}

\section{Morphisms from $X_0(p)$ to curves of genus $\geq2$}\label{morphism_section}

\begin{definition}
    Let $B$ be a polarized abelian variety with polarization $\phi_B : B \to B^\vee$ and $f : A \hookrightarrow B$ an abelian subvariety. Then $\phi_A :=f^\vee\circ \phi_B \circ f : A \to A^\vee$ is called the induced polarization on $A$.
\end{definition}

\begin{lemma}\label{polarization_divisibility_lemma}
Let $f: X \to Y$ be a morphism of nice curves over $\C$ with $g(Y) \geq 1$. Let $A := f^*(J(Y)) \subseteq J(X)$ be the image of $J(Y)$ in $J(X)$. Let $\phi_A : A \to A^\vee$ be the polarization on $A$ induced by the principal polarization on $J(X)$. Then $\deg f$ is divisible by the exponent of $\ker \phi_A$.
\end{lemma}
\begin{proof}
We let $d:=\deg f$ for simplicity. Since $A=f^*(J(Y))$, we can write $f^*=f_2\circ f_1$ with $f_1 : J(Y) \to A$ being just $f^*$ with restricted codomain and $f_2: A \hookrightarrow J(X)$ being the inclusion morphism. Then $f_* =f_1^\vee\circ f_2^\vee$ and we have the following equality of endomorphisms of $J(Y)$:  $$[d] = f_*\circ f^* = f_1^\vee\circ f_2^\vee \circ f_2\circ f_1= f_1^\vee\circ \phi_A \circ f_1.$$
Applying the kernel-cokernel exact sequence to the decomposition $[d] = (f_1^\vee\circ \phi_A) \circ f_1$ and using the surjectivity of $f_1$ we get a surjection $J(Y)[d] \xrightarrow{f_1} \ker (f_1^\vee \circ \phi_A)$.  Together with the inclusion $\ker (\phi_A) \subseteq \ker (f_1^\vee \circ \phi_A)$, this gives the desired divisibility $$\exp(\ker \phi_A)  \mid  \exp(\ker (f_1^\vee \circ \phi_A)) \mid  \exp(J(Y)[d]) =d.$$

\end{proof}

\begin{remark}
If $A \subseteq J_0(N)$, then $\ker \phi_A$ is called the modular kernel \cite[Section 10.6]{stein07} of $A$ and can be computed using SageMath. 
\end{remark}

\begin{corollary}\label{cor:not_curve_criterion}
Let $X$ be a nice curve over $\C$ of genus $g$ and let $A \subseteq J(X)$ be an abelian subvariety of dimension $g'\geq 2$. Assume that $2g-2 < \exp(\ker \phi_A) (2g'-2)$. Then there does not exist a morphism $f : X \to Y$ of nice curves such that $f^*(J(Y)) = A$.
\end{corollary}
\begin{proof}
    Suppose that we have such a morphism $f:X\to Y$. Then we must have $g(Y)= g'$. This is because $\dim J(Y)=g(Y)$, $f^*(J(Y)) = A$, and because $f^*$ has finite kernel (it is contained in the kernel of $[d]$ since $[d]=f_*\circ f^*$). Therefore $A$ and $J(Y)$ must have the same dimension. We also get $\exp(\ker \phi_A) \mid \deg f$ from \Cref{polarization_divisibility_lemma}. However, the Riemann-Hurwitz formula gives us \[2g-2\geq\deg f\cdot(2g(Y)-2)\geq \exp(\ker \phi_A)\dot(2g'-2),\] which is a contradiction with the assumption in this corollary.
\end{proof}

We now present the proof of \Cref{thm_p<3000}. The crucial result we use there is the following lemma.

\begin{lemma}\label{J0p_finite_subvarieties}
    Let $p$ be a prime. Then the Jacobian $J_0(p)$ has only finitely many abelian subvarieties over $\Q$.
\end{lemma}

\begin{proof}
    Since $p$ is prime, every simple abelian variety in $J_0(p)_\Q$ is new and therefore occurs with multiplicity $1$ by \cite[Theorem 9.4]{stein07} and \cite[Theorem 5.8.3]{modular}. Thus \[J_0(p)_\Q \sim A_1\oplus\ldots\oplus A_k\] with the $A_i$ pairwise non-isogenous simple abelian varieties over $\Q$. If $A \subseteq J_0(p)$ is an abelian subvariety, then, since all the $A_i$ occur with multiplicity $1$, there will be a finite subset $S \subseteq \set{1,\ldots,k}$ such that $A =\footnote{We actually mean equality here, not just isogenous.} \sum_{i \in S} A_i$. In particular, $J_0(p)$ will have exactly $2^k$ different abelian subvarieties over $\Q$.
\end{proof}

\begin{remark}\label{J0N_infinite_subvarieties}
    This result does not hold for composite levels $N$. We have a decomposition \[J_0(N)_\Q \sim A_1^{n_1}\oplus\ldots\oplus A_k^{n_k},\] where the level of $A_i$ is $M_i\mid N_i$ and $n_i=\textup{number of divisors of }\frac{N_i}{M_i}$. Hence, for all but finitely many composite levels $N$ the Jacobian $J_0(N)$ will have isogeny factors of multiplicity $\geq2$. However, if $A$ is a simple abelian variety over $\Q$ and $n\geq2$, then $A^n$ has infinitely many abelian subvarieties over $\Q$ isogenous to $A$. 
\end{remark}

\begin{proof}[Proof of \Cref{thm_p<3000}]
By \Cref{J0p_finite_subvarieties}, there exist pairwise non-isogenous simple abelian varieties $A_1,\ldots,A_k$ over $\Q$ such that $J_0(p)_\Q \sim A_1\oplus\ldots\oplus A_k$ and $J_0(p)$ has exactly $2^k$ different abelian subvarieties over $\Q$. Using SageMath we enumerated all abelian subvarieties $A \subseteq J_0(p)$ for $p \leq \computationbound$.

For each of the $2^k$ subvarieties $A$ we check the possibilities for the curve $Y$ and morphism $f$. We divide the proof into cases depending on $A:=f^*(J(Y))$. Let $g$ be the genus of $X_0(p)$ and $g'$ be the dimension of $A$. Notice that $g\geq g(Y)\geq2$.

{\bf Case 1:} $2g-2 < \exp(\ker \phi_A) (2g'-2)$. In this case the abelian variety $A$ does not come from a curve because of \Cref{cor:not_curve_criterion}.

{\bf Case 2:} $2g-2 < 3(2g'-2)$. As mentioned before, we must have $g(Y)=g'$. The Riemann-Hurwitz formula gives us \[2g-2\geq\deg f\cdot(2g'-2)\] which forces $\deg f =2$. This means that there is an involution $\iota$ of $X_0(p)$ such that $Y = X_0(p)/\iota$. By \cite[Theorem 0.1]{kenkumomose1988automorphism}, for $p\neq37$ the only non-trivial automorphism of $X_0(p)$ is $w_p$. Furthermore, the automorphism group of $X_0(37)$ is equal to $(\Z/2\Z)^2$ and the three quotient curves $X_0(37)/\iota$ are $\PP^1$ and two elliptic curves (LMFDB labels \lmfdbec{37}{a}{1}, \lmfdbec{37}{b}{2}). This shows that $\iota=w_p$, and hence $f$ is isomorphic to the quotient map $X_0(p) \to X_0^+(p) := X_0(p)/w_p$.

{\bf Case 3:} $A=\pi^*(J_0(p)^+)$ where $\pi : X_0(p) \to X_0^+(p)$ is the quotient map. Since $A=f^*(J(Y))=\pi^*(J_0(p)^+)$, \cite[Proposition 2.11 (i)]{BakerGonzalezPoonen2005} shows that in this case there exist non-constant morphisms $u_1:Y\to X_0^+(p), u_2:X_0^+(p)\to Y$ such that $\pi=u_1\circ f, f=u_2\circ\pi$. Hence $\pi=u_1\circ u_2\circ\pi$, implying that $\deg (u_1\circ u_2)=1$, so $f$ and $\pi$ are isomorphic. 

{\bf Case 4:} The remaining cases. These cases are listed in \Cref{tab:remaining}. In the last column we give the Riemann-Hurwitz bound on $\deg f$ for expositional reasons.

This table has been computed using the \texttt{lower\_genus\_candidates\_parallel(N)} function with $N=3000$ (available on GitHub). This function runs over all primes $<N$ and discarding those falling under cases 1 to 3. The output of this Sage program can be found in the file \texttt{logs/run.txt}.

\begin{center}
\begin{longtable}{|c|c|c|c|c|c|}
\caption{Abelian subvarieties $A\subseteq J_0(p)$ not falling into the first $3$ cases.}
\label{tab:remaining}\\

\hline
$p$ & $g:=g(X_0(p))$ & $g(X_0^+(p))$ & $g':=\dim A$ & $\exp(\ker \phi_A)$ & $\left\lfloor\displaystyle\frac{2g-2}{2g'-2}\right\rfloor$\\
    \hline

$223$ & $18$ & $6$ & $2$ & $14$ & $17$\\
$227$ & $19$ & $5$ & $2$ & $14$ & $18$\\
$359$ & $30$ & $6$ & $2$ & $16$ & $29$\\
$383$ & $32$ & $8$ & $2$ & $22$ & $31$\\
$491$ & $41$ & $12$ & $2$ & $38$ & $40$\\
$809$ & $67$ & $26$ & $2$ & $24$ & $66$\\
$929$ & $77$ & $30$ & $2$ & $40$ & $76$\\
$1409$ & $117$ & $50$ & $2$ & $48$ & $116$\\
 
    \hline

\end{longtable}

\end{center}

As can be seen from the table, we have $g':=\dim A=2$ in all these cases. Therefore, if $A$ came from a curve $Y$ we must also have $g(Y)=g'=2$ (by the proof of \Cref{cor:not_curve_criterion}) and hence $Y$ is a hyperelliptic curve. We can now check in \cite[Tables 1,2]{GonzalezJimenez2002} and \cite[Theorem 1.3, Table 4]{BakerGonzalezPoonen2005} that there are no non-constant morphisms from $X_0(p)$ to a genus $2$ curve for these levels $p$.

\end{proof}

\begin{corollary}\label{thm_p_plus<3000}
Let $p<\computationbound$ be a prime. Then the quotient curve $X_0^+(p)$ does not admit a non-constant rational morphism to a curve of genus $\geq 2$.
\end{corollary}

\begin{remark}
The proof of \Cref{thm_p<3000} uses computational methods and is based on the ability to enumerate all abelian subvarieties of $J_0(p)$. This heavily uses the fact that $p$ is prime because all simple abelian varieties $A$ occur with multiplicity $1$ in $J_0(p)$.

The computationally hardest level was $p=2089$. In that case the Jacobian $J_0(2089)$ has $10$ simple factors. The decomposition is
\begin{align*}
    J_0(2089)_\Q \sim & \ E_{1,\dim=1}\oplus E_{2,\dim=1}\oplus E_{3,\dim=1}\oplus E_{4,\dim=1}\oplus E_{5,\dim=1}\oplus\\
    &\oplus A_{1,\dim=2},\oplus A_{2,\dim=2}\oplus A_{3,\dim=6}\oplus A_{4,\dim=67}\oplus A_{5,\dim=91},
\end{align*}
as can be seen on LMFDB \cite{lmfdb}. \begin{center}
    \url{https://www.lmfdb.org/ModularForm/GL2/Q/holomorphic/?level=2089&weight=2&char_order=1}
\end{center} However, we can see that two of those factors are of dimension $67$ and $91$, respectively, and can therefore be eliminated as in Case 2 of the proof. The computation for $p=2089$ took around $4$ days. The other levels had fewer simple factors and the computations were much faster, no more than several minutes.

For composite levels $N$ the Jacobian $J_0(N)$ will often contain simple factors that occur with higher multiplicity. If $J_0(N)$ contains an isogeny factor $A^n$ with multiplicity $n>1$, then, in fact, $J_0(p)$ will contain infinitely many distinct abelian subvarieties isogenous to $A$. However, the computations in the proof of \Cref{thm_p<3000} could easily be generalized to composite levels $N$ if we restrict ourselves to modular parametrizations $f : X_0(N) \to Y$ that are new of level $N$, i.e., those for which $f^*(J(Y))\subseteq J_0(N)_{\textup{new}}$, see \cite[Introduction]{BakerGonzalezPoonen2005} for further results on new modular curves. 
\end{remark}

We now move on to the proof of \Cref{thm_jacobian_decomposition}. First, we present several auxiliary results regarding the genus of the quotient curve $X_0^+(N)$. From now on, we will denote the genus of $X_0(N)$ by $g_0(N)$ and the genus of $X_0^+(N)$ by $g_0^+(N)$ for simplicity.

\begin{lemma}[{\cite[Section 2]{JEON2018319}, \cite[Remark 2]{FurumotoHasegawa1999}}]\label{quotient_genus_formula}
    The genus of the quotient curve $X_0^+(N)$ is equal to \[g_0^+(N)=\frac{1}{4}(2g_0(N)+2-\alpha_Nh(-4N)),\] where $h(-4N)$ is the class number of primitive quadratic forms of discriminant $-4N$ and \[\alpha_N= \begin{cases}
        & 2, \textup{ if } N\equiv 7\pmod{8},\\
        & \frac{4}{3}, \textup{ if } N\equiv 3\pmod{8} \textup{ and } N>3,\\
        & 1, \textup{ otherwise}.
    \end{cases}\]
\end{lemma}

\begin{lemma}[{\cite[Corollary 7.28]{cox}, see also \cite[Section A.1]{DerickxNajman2023}}]\label{cox_class_number}
    Let $d<0$ be a fundamental discriminant and $m$ a positive integer. Then we have \[h(dm^2)=\frac{h(d)m}{w_{d,m}}\prod_{p\mid m}\left(1-\left(\frac{d}{p}\right)\frac{1}{p}\right),\] where $\left(\frac{d}{p}\right)$ is the Kronecker symbol, $w_{d,m}=3$ if $d=-3$ and $m\neq1$, $w_{d,m}=2$ if $d=-4$ and $m\neq1$, and $w_{d,m}=1$ otherwise.
\end{lemma}

\begin{lemma}[{\cite[Corollary 1]{Ramar2001}}]\label{class_number_upper_bound}
    Let $d<-4$ be a fundamental discriminant. Then we have \[h(d)\leq\frac{\sqrt{|d|}}{2\pi}(\log |d|+5-2\log6).\]
\end{lemma}

\begin{corollary}\label{class_number_ub_corollary}
    Let $p\geq5$ be a prime. Then we have \[h(-4p)\leq
    \frac{3\sqrt{p}}{2\pi}(\log p+5-2\log 6).\]
\end{corollary}

\begin{proof}
    If $p\equiv3\pmod{4}$, then $-p$ is a fundamental discriminant. By \Cref{cox_class_number} we have $h(-4p)=2h(-p)(1\pm\frac{1}{2})\leq 3h(-p)$ and by \Cref{class_number_upper_bound} we have \[h(-4p)\leq 3h(-p)\leq\frac{3\sqrt{p}}{2\pi}(\log p+5-2\log6).\]
     If $p\equiv1\pmod{4}$, then $-4p$ is a fundamental discriminant and we can apply \Cref{class_number_upper_bound} to get \[h(-4p)\leq\frac{\sqrt{4p}}{2\pi}(\log 4p+5-2\log 6)\leq \frac{3\sqrt{p}}{2\pi}(\log p+5-2\log6)\]
     as well. It can be easily checked that the last inequality holds for all $p\geq5$.
\end{proof}

\begin{remark}
    Since $\prod_{p\mid m}\left(1-\left(\frac{d}{p}\right)\frac{1}{p}\right)\leq\prod_{p\mid m}\left(1+\frac{1}{p}\right)<c\log m$ for some computable constant $c>0$, it is possible to obtain a similar upper bound on $h(d)$ for general $d<0$.
\end{remark}

\begin{proof}[Proof of \Cref{thm_jacobian_decomposition}]
    We may suppose that $p>\computationbound$ due to \Cref{thm_p<3000}. We define $g_0(p):=g(X_0(p))$. From \Cref{quotient_genus_formula} and \Cref{class_number_ub_corollary} we conclude that \begin{align*}
        \frac{1}{4}(2g_0(p)+2-2h(-4p))&\leq g_0^+(p)\leq \frac{1}{4}(2g_0(p)+2-h(-4p))\\
        \implies \frac{g_0(p)}{2}+\frac{1}{2}-\frac{3\sqrt{p}}{4\pi}(\log p+5-2\log 6)&\leq g_0^+(p)\leq\frac{g_0(p)}{2}.
    \end{align*}

    Moreover, if $p\geq5$ is a prime, then the genus formula for $X_0(p)$ is \[g_0(p)=\begin{cases}
        \displaystyle\frac{p-13}{12} &: p\equiv 1\pmod{12},\\ \
        \displaystyle\frac{p-5}{12} &: p\equiv 5\pmod{12},\\ \
        \displaystyle\frac{p-7}{12} &: p\equiv 7\pmod{12},\\ \
        \displaystyle\frac{p+1}{12} &: p\equiv 11\pmod{12}.
    \end{cases}\] It can easily be derived from the genus formula for $X_0(N)$, see \cite[Theorem 3.1.1, Exercise 3.1.4 (e)]{modular}. From this formula it follows that $g_0(p)\geq\frac{p-13}{12}$. Therefore, for $p>43644$ we have that $\frac{g_0(p)}{3}+3<\frac{g_0(p)}{2}+\frac{1}{2}-\frac{3\sqrt{p}}{4\pi}(\log p+5-2\log 6)$, meaning that \[\frac{g_0(p)}{3}+3<g_0^+(p)\leq\frac{g_0(p)}{2}.\]

    For all primes $\computationbound<p<43644$ we computed $g_0^+(p)$ manually using Sage and obtained the same inequality $\frac{g_0(p)}{3}+3<g_0^+(p)\leq\frac{g_0(p)}{2}$.

    Since $\dim A\leq 2$, this implies that $\frac{g_0(p)}{3}+1<\dim A^+\leq\frac{g_0(p)}{2}$ and $\dim A^-\geq\frac{g_0(p)}{2}-2$. Suppose now that we have a rational morphism $f:X_0(p)\to Y$ of degree $d>1$. We divide the proof into several cases.

    {\bf Case 1:} $A^+\subseteq f^*(J(Y))$. From the proof of \Cref{cor:not_curve_criterion} we conclude that $g(Y)\geq\dim A^+>\frac{g_0(N)}{3}+1$. The Riemann-Hurwitz formula gives us \[2g_0(N)-2\geq\deg f\cdot(2g(Y)-2)>\deg f\cdot\frac{2g_0(N)}{3}\] which forces $\deg f=2$. We now use the same argument as in Case 2 of the proof of \Cref{thm_p<3000} to show that $f$ is isomorphic to the quotient map $X_0(p)\to X_0^+(p)$.

    {\bf Case 2:} $A^-\subseteq f^*(J(Y))$. From the proof of \Cref{cor:not_curve_criterion} we conclude that $g(Y)\geq\dim A^->\frac{g_0(N)}{2}-2$. The Riemann-Hurwitz formula gives us \[2g_0(N)-2\geq\deg f\cdot(2g(Y)-2)>\deg f\cdot(g_0(N)-6)\] which forces $\deg f=2$. Now $f$ must be isomorphic to the quotient map $X_0(p)\to X_0^+(p)$, but this is impossible because then $A^-\nsubseteq f^*(J(Y))$.

    {\bf Case 3:} $f^*(J(Y))\subseteq A$. We conclude that $g(Y)\leq \dim A\leq 2$. Since $g(Y)\geq2$ by the assumption of the theorem, we conclude that $g(Y)=\dim A=2$ and $Y$ must be a hyperelliptic curve. Similarly as in Case 4 of the proof of \Cref{thm_p<3000}, we can now check in \cite[Tables 1,2]{GonzalezJimenez2002} and \cite[Theorem 1.3, Table 4]{BakerGonzalezPoonen2005} that there are no non-constant morphisms from $X_0(p)$ to a genus $2$ curve when $p$ is prime.
\end{proof}


\section{Modular curves $X_0(p)$ with infinitely many degree $6$ points}\label{inf_deg6points_section}

In this section we prove that all curves $X_0(p)$ listed in \Cref{deg6_thm} have infinitely many degree $6$ points over $\Q$. The sources of these infinitely many degree $6$ points will be the pullbacks of infinitely many lower degree points by a rational morphism.

\begin{proposition}\label{deg6map}
    The modular curve $X_0(p)$ has infinitely many degree $6$ points over $\Q$ for $p\leq 61$ and \[p\in\{71,79,83,89,97,101,113,127,131,137,139,149,151,179,181,227,239\}.\]
\end{proposition}
\begin{proof}
    First, let us prove that all these curves $X_0(p)$ admit a degree $6$ rational morphism to $\PP^1$. For $p\leq47$ and $p\in\{59,71\}$ the curve $X_0(p)$ has $\Q$-gonality $\leq3$ by \cite{Ogg74, HasegawaShimura1999}. The morphism $a^d: \PP^1\to \PP^1, (x:y)\mapsto(x^d:y^d)$ is defined over $\Q$ and is of degree $d$. Let $f: X_0(p)_\Q \to \P^1$ be a gonal map and $d:=\displaystyle\frac{6}{\textup{gon}_\Q(X_0(p))}$, then the composition $X_0(p)\xrightarrow{f}\PP^1\xrightarrow{a^d}\PP^1$ is of degree $6$ and defined over $\Q$.
    
    For \[p\in\{53,61,79,83,89,101,131\}.\] the quotient curve $X_0^+(p)$ is an elliptic curve. Any elliptic curve $E: y^2+a_1xy+a_3y=x^3+a_2x^2+a_4x+a_6$ admits a degree $3$ rational morphism $y: E\to\PP^1$. Therefore, the composition map $X_0(p)\to X_0^+(p)\to\PP^1$ is of degree $6$ and defined over $\Q$.

    For \[p\in\{97,113,127,137,139,149,151,179,181,227,239\}.\] the curve $X_0(p)$ has $\Q$-gonality equal to $6$ by \cite[Theorem 1.4]{NajmanOrlic22}.

    The existence of this degree $6$ rational morphism $X_0(p)\to\PP^1$ together with Hilbert's Irreducibility Theorem \cite[Proposition 3.3.5(2) and Theorem 3.4.1]{serre2016topics} implies the existence of infinitely many degree $6$ points on $X_0(p)$.
\end{proof}

\begin{proposition}\label{bars_cubic_points}
    The modular curve $X_0(p)$ has infinitely many degree $6$ points over $\Q$ for \[p\in\{67,73,103,107,109, 163,167,191,269\}.\]
\end{proposition}

\begin{proof}
 The quotient curve $X_0^+(p)$ has infinitely many cubic points for all these levels $p$ by \cite[Theorem 1]{BarsDalal22}. The inverse images of these cubic points under the degree $2$ quotient map $X_0(p) \to X_0^+(p)$ are points of degree $3$ or degree $6$ over $\Q$.  
 Since these curves $X_0(p)$ have only finitely many points of degree $3$ by \Cref{thmcubic}, we conclude that $X_0(p)$ must have infinitely many degree $6$ points.
\end{proof}

\section{Modular curves $X_0(p)$ with finitely many degree $6$ points}\label{finite_deg6points_section}

In this section we use the results from \Cref{morphism_section} to prove that all curves $X_0(p)$ for prime levels $p$ not listed in \Cref{deg6_thm} have only finitely many degree $6$ points over $\Q$. We use several methods to do that. The first step is to give an upper bound on $p$ using \Cref{frey}.

\begin{lemma}[Ogg, \cite{Ogg74}]\label{Ogg} For a prime $q$ with $q\nmid N$, let $\#X_0(N)(\F_{q^2})$ denote the number of $\F_{q^2}$-rational points on $X_0(N)$.
Then $$\#X_0(N)(\F_{q^2})\geq L_q(N):=\frac{q-1}{12} \psi(N)+2^{\omega(N)},$$
where $\psi(N)=N\displaystyle\prod_{r|N\textup{ prime}}\left(1+\frac{1}{r}\right)$ and $\omega(N)$ is the number of distinct prime factors of $N$.
\end{lemma}

We note that \Cref{Ogg} is not explicitly proven in \cite{Ogg74}, and one can refer to \cite[Lemma 3.20]{BakerGonzalezPoonen2005} for the explicit proof.

\begin{corollary}\label{ogg_cor}
Let $p$ be a prime, then the modular curve $X_0(p)$ has only finitely many degree $d$ points over $\Q$ for $p > 120d-24$.
\end{corollary}

\begin{proof}
    Suppose that $X_0(p)$ has infinitely many degree $d$ points over $\Q$. \Cref{frey} tells us that in that case the $\Q$-gonality of $X_0(p)$ must be at most $2d$. Let us take the rational morphism $f:X_0(p)\to\PP^1$ of degree $\leq 2d$.

    Since $p>2$, the curve $X_0(p)$ has good reduction at $q=2$ and hence $f$ induces a map $f_{\F_2} : X_0(p)_{\F_2}\to  \PP^1$ of degree $\leq 2d$ (cf. \cite[Proof of Theorem 3.2]{NguyenSaito}). Now we know that there are at most $10d$ $\F_4$-rational points on $X_0(p)_{\F_2}$ since at most $2d$ points can map to each of the $5$ $\F_4$-rational points on $\PP^1$. \Cref{Ogg} now gives us the inequality \[\frac{p+1}{12}+2\leq\#X_0(p)(\F_{4})\leq10d,\] implying that $p < 120d-24$.
\end{proof}

This leaves us with reasonably many levels $p$ we need to check. First, we will deal with the curves of genus $\geq 17$ using the results of \cite{KadetsVogt}. The main topic of that article is the minimum density degree $\delta(C/K)$ and we will use several results from there in our study of degree $6$ points on $X_0(p)$.

\begin{definition}\cite[Page 1]{KadetsVogt}
    Let $C$ be a curve defined over a number field $K$. The {\it density degree set} $\delta(C/K)$ is the set of integers $d$ for which the collection of closed points of degree $d$ on $C$ is infinite. The {\it minimum density degree}\footnote{In previous versions of \cite{KadetsVogt}, the minimum density degree was called \textit{the arithmetic degree of irrationality} $\textup{a.irr}_K C$.} $\min(\delta(C/K))$ is the smallest integer $d\geq1$ such that $C$ has infinitely many closed points of degree $d$ over $K$, or equivalently:
    $$\min(\delta(C/K)) =\min\left\{d 
\in \N \mid \#\left\{\cup_{[F:K]=d} C(F)\right\}=\infty\right\}.$$
    We also define the {\it potential density degree set} $\wp(C/K):=\bigcup_{[L:K]<\infty}\delta(C/L)$, and analogously the {\it minimum potential density degree}
    $$\min(\wp(C/K)) :=\min \bigcup_{[L:K]<\infty}\delta(C/L).$$
\end{definition}

\begin{theorem}[{\cite[Theorem 1.3]{KadetsVogt}}]\label{KVthm}
Suppose that $C$ is a nice curve of genus $g$ over a number field $K$ and $\min(\delta (C/K))=d.$ Let $m:=\lceil{d/2}\rceil -1$ and let $\epsilon := 3d-1-6m < 6.$ Then one of the following holds:
\vskip 0.1in
(1) There exists a $K$-rational non-constant morphism of curves $\phi \colon C \rightarrow Y$ of degree at least $2$ such that $d = \min(\delta(Y/K)) \cdot \deg \phi.$
\vskip 0.1in
(2) $g \leq \max \left(\frac{d(d-1)}{2} + 1, 3m (m-1)+m \epsilon \right).$
\end{theorem}

Note that the above theorem can be used to get lower bounds on the minimum density degree using the following corollary. This corollary can be easily proved by applying \Cref{KVthm} for all degrees $\leq d$.

\begin{corollary}\label{KVthmv2}
Suppose that $C$ is a nice curve of genus $g$ over a number field $K$. Let $d$ be an integer, $m:=\lceil{d/2}\rceil -1$ and let $\epsilon := 3d-1-6m < 6.$ Assume that:
\begin{enumerate}
    \item 
 For all $K$-rational non-constant morphisms of curves $\phi \colon C \rightarrow Y$ of degree at least $2$ one has $ \min(\delta(Y/K)) \cdot \deg \phi> d,$ and
\item $g > \max \left(\frac{d(d-1)}{2} + 1, 3m (m-1)+m \epsilon \right).$
\end{enumerate}
Then $\min(\delta (C/K))>d.$
\end{corollary}

The above result shows that having good knowledge about rational morphisms $\phi : X_0(p) \to Y$ allows one to deduce a lower bound on $\min(\delta (C/K))$. 
The following theorem shows that if $d \leq 25$, then the degree $2$ map $X_0(p) \to X_0^+(p)$ is the only morphism with $g(Y) \geq 2$ for which condition $(1)$ of \Cref{KVthmv2} could potentially fail. It reduces the verification of the $g(Y) \geq 2$ part of condition $(1)$ to proving a lower bound of size $d/2$ on $\min(\delta(X_0^+(p)/\Q))$.

\begin{theorem}\label{thm:maps_for_small_degree_points}
    Let $d\leq 25$ be an integer and $p$ be a prime. Suppose that there exists a non-constant rational morphism of nice curves $\phi \colon X_0(p) \rightarrow Y$ of degree at least $2$ such that
    \begin{enumerate}
        \item $Y$ has genus $\geq 2$ and
        \item $Y$ has infinitely many points of degree $\leq d/\deg \phi$.
    \end{enumerate}
    Then, up to isomorphism, one has that $Y=X_0^+(p)$ and $\phi:X_0(p) \to X_0^+(p)$ is the quotient map.
\end{theorem}
\begin{proof}
If $p<\computationbound$, then this is a consequence of \Cref{thm_p<3000}. Now suppose that $p\geq\computationbound$. 
Since $(p+24)/120 > d$, one has that $X_0(p)$ has only finitely many points of degree $\leq d$ due to \Cref{ogg_cor}. The result now follows since this implies there are no $\phi$ satisfying assumption (2). 
\end{proof}

Therefore, if one wants to search for infinitely many points of degree $d\leq 25$ on curves $X_0(p)$ as pullbacks of smaller degree points, the only possibilities are morphisms $X_0(p)\to\PP^1$, $X_0(p)\to E$, and the quotient map $X_0(p)\to X_0^+(p)$.

For $\PP^1$ this becomes the question of determining the $\Q$-gonality. For elliptic curves $E$ we know that the degree of a morphism to $E$ must be a multiple of the modular degree of $E$ ($p$ is a prime, so the conductor of $E$ must be equal to $p$). The LMFDB contains all elliptic curves over $\Q$ of conductor at most $500,000$ which is enough for all $d\leq \floor{500024/120} = 4166$.

\begin{remark} For smaller degrees $d$ \Cref{KVthm} can be very useful for showing that $X_0(p)$ only has finitely many points of degree $d$. However, if one wants to study the points of degree $d>25$ on $X_0(p)$, then one cannot get new information from \Cref{KVthm}. Indeed, if $g\leq\frac{d(d-1)}{2}+1$, then we cannot apply the theorem. From the genus formula for $X_0(p)$ we know that $\frac{p-13}{12}\leq g_0(p)\leq\frac{p+1}{12}$. So if $g_0(p)\geq\frac{d(d-1)}{2}+2$, then $p \geq 6d(d-1)+23$. Now, if $d \geq 21$, then $6d(d-1)+23 > 120d-24$ so we already know that $X_0(p)$ has only finitely many degree $d$ points by \Cref{ogg_cor}. Therefore, there is no merit in using \Cref{KVthm} for such $d$.
\end{remark}

When $Y= X_0^+(p)$, we need to bound $\min(\delta(X_0^+(p)/\Q))$ from below. To this end, the following variation of \Cref{thm:maps_for_small_degree_points} verifies a large part of condition (1) of \Cref{KVthmv2}.

\begin{theorem}\label{thm:maps_for_small_degree_points_plus}
    Let $d\leq 12$ be an integer and $p$ be a prime. Then there does not exist any non-constant rational morphism of nice curves $\phi \colon X_0^+(p) \rightarrow Y$ of degree at least $2$ such that
    \begin{enumerate}
        \item $Y$ has genus $\geq 2$ and
        \item $Y$ has infinitely many points of degree $\leq d/\deg \phi$.
    \end{enumerate}
\end{theorem}
\begin{proof}
If $p<\computationbound$, then this is a consequence of \Cref{thm_p_plus<3000}. Now suppose that $p\geq\computationbound$. 
Since $(p+24)/120 > 2d$, one has that $X_0(p)$ has only finitely many points of degree $\leq 2d$ due to \Cref{ogg_cor} and hence $X_0^+(p)$ has only finitely many points of degree $\leq d$. The result now follows since this implies that there are no morphisms $\phi$ satisfying the assumption (2).
\end{proof}

The following corollary of \Cref{KVthm} gives conditions for general curves $C/\Q$ with infinitely many degree $6$ points.

\begin{corollary}\label{KVcor}
If $C/\Q$ is a nice curve such that $C(\Q)\neq\emptyset$, $\min(\delta(C/\Q))=6$, and $g(C) \geq 17,$ then one of the following holds:
\vskip 0.1in
(1) $C$ admits a $\Q$-rational morphism of degree $6$ to either $\PP^1$ or an elliptic curve $E$ of positive rank over $\Q$.
\vskip 0.1in
(2) $C$ admits a $\Q$-rational morphism of degree $2$ to a genus $4$ curve with infinitely many cubic points over $\Q$.
\end{corollary}
\begin{proof}
Since $g(C) \geq 17$, case (2) of \Cref{KVthm} cannot occur. This means that there exists a non-constant morphism of curves $\phi \colon C \rightarrow Y$ of degree at least $2$ over $\mathbb{Q}$ with the property that $6= \min(\delta(Y/\Q)) \cdot \deg \phi.$

If $\deg\phi=6$, then $Y$ has infinitely many rational points and hence it is either $\PP^1$ or an elliptic curve $E$ of positive rank over $\Q$ by Faltings' theorem.

If $\deg\phi=3$, then $Y$ has infinitely many quadratic points. Hence $Y$ admits a degree $2$ rational morphism to $\PP^1$ or to an elliptic curve $E$ of positive rank over $\Q$ by \Cref{harrissilverman} and \cite[Theorem 1.2 (1)]{KadetsVogt}. Therefore, the desired map of degree $6$ to $\P^1$ or $E$ can be obtained as a composition $C\to Y\to(\PP^1 \textup{ or } E)$.

If $\deg\phi=2$, then $\min(\delta(Y/\Q))=3$. \cite[Theorem 1.2 (2)]{KadetsVogt} now tells us that there are $3$ cases regarding the curve $Y$:

\begin{enumerate}
    \item $Y$ admits a $\Q$-rational morphism of degree $3$ to either $\PP^1$ or an elliptic curve $E$ of positive rank over $\Q$.
    \item $Y$ is a smooth plane quartic with no rational points, positive rank Jacobian, and at
least one cubic point.
\item $Y$ is a genus $4$ Debarre–Fahlaoui curve (see the precise definition of Debarre-Fahlaoui curves in \cite[Definition 5.1]{KadetsVogt}).
\end{enumerate}

The second case is impossible since $C(\Q)\neq\emptyset$ implies that $Y(\Q)\neq\emptyset$. In the first case we have a degree $6$ rational composition map $C\to Y\to(\PP^1 \textup{ or } E)$, and in the third case $Y$ is a curve of genus $4$ with infinitely many cubic points.
\end{proof}

\subsection{Proof of \Cref{thm_jacobian_decomposition}}

\begin{proposition}\label{KVhighgenus}
    Let $p > 200$ be a prime. Then the modular curve $X_0(p)$ has infinitely many degree $6$ points over $\Q$ if and only if $p \in \set{227, 239, 269}$.
\end{proposition}
\begin{proof}
Note that for $p>696$ this was already proved by \Cref{ogg_cor}. For $p=227, 239, 269$ the existence of infinitely many degree $6$ points was proved in \Cref{deg6map} and \Cref{bars_cubic_points}. For the remaining primes $p$ between $200$ and $696$ we argue as follows.

The remaining curves $X_0(p)$ have genus $g\geq17$. To prove that, we can use the formula for the genus of $X_0(N)$ \cite[Proposition 1.43]{Shimura:AutomorphicForms} \[g(X_0(N))=1+\frac{\psi(N)}{12}-\frac{\nu_2}{4}-\frac{\nu_3}{3}-\frac{\nu_\infty}{2}.\] Here $\psi(N)=N\displaystyle\prod_{r|N\textup{ prime}}\left(1+\frac{1}{r}\right)$, $\nu_2$ and $\nu_3$ are the numbers of solutions in $\Z/N\Z$ of the equations $x^2+1=0$ and $x^2+x+1=0$, respectively, and $\nu_\infty=\sum_{d\mid N} \varphi(d,N/d)$ is the number of cusps on $X_0(N)$. For $p=211$ we can calculate that $g=17$ manually using this formula and for $p\geq223$ we get the inequality \[g(X_0(p))=1+\frac{p+1}{12}-\frac{\nu_2}{4}-\frac{\nu_3}{3}-1\geq\frac{p+1}{12}-\frac{2}{4}-\frac{2}{3}=\frac{p-13}{12}>17.\]
   We now apply \Cref{KVthmv2} with $d=6$. Condition $(2)$ of that corollary is satisfied since the remaining $X_0(p)$ have genus $\geq 17$. We claim that condition $(1)$ is also satisfied. Indeed, let $\phi : C\to Y$ be a non constant morphism of degree $\geq 2$. 
   
   If $Y$ has genus $0$, then $Y\cong \P^1$ since $X_0(p)$ has rational points corresponding to the cusps. This implies that $\min(\delta(X_0(p)/\Q)) \deg(\phi) \geq \deg(\phi) \geq 7$ since all these curves have gonality $\geq 7$ by \cite[Theorem 1.1]{NajmanOrlic22}.
   
   If $Y$ has genus $1$, then it is an elliptic curve since $X_0(p)$ has a rational point. Condition $(1)$ of \Cref{KVthmv2} is satisfied since the only elliptic curve of prime conductor between $200$ and $696$ and modular degree $\leq 6$ is the elliptic curve \lmfdbec{269}{a}{1} according to the LMFDB, see:
   
   \url{https://www.lmfdb.org/EllipticCurve/Q/?conductor_type=prime&conductor=200-696&sort_order=degree&showcol=degree&sort_order=degree}.  
   
   Finally, if $Y$ has genus $\geq 2$, then the condition $(1)$ of \Cref{KVthmv2} is also satisfied because of \Cref{thm:maps_for_small_degree_points} and the fact that $X_0^+(p)$ has only finitely points of degree $\leq 3$ in the remaining cases due to \cite[Theorems 1 and 3]{BarsDalal22}.

   Therefore, we can apply \Cref{KVthmv2} to conclude $\min(\delta(X_0(p)/\Q)>6$. Hence $X_0(p)$ has only finitely many degree $6$ points in the remaining cases.
\end{proof}

\begin{remark}\label{deg2map_remark}
    We can deal with the case $g(Y)\geq2$ in another way using \Cref{KVcor}. Suppose that we have a degree $2$ rational morphism $f:X_0(p)\to Y$. Then $f$ corresponds to an involution on $X_0(p)$. However, if $N\neq37,63$, then the only non-trivial involutions on curves $X_0(N)$ of genus $g\geq2$ are Atkin-Lehner involutions by \cite[Theorem 0.1]{kenkumomose1988automorphism}. Therefore, up to isomorphism, the only degree $2$ rational morphism from $X_0(p)$ to a curve of genus $g\geq2$ is the quotient map $X_0(p)\to X_0^+(p)$.

    However, this argument with involutions will not work for morphisms of degree $\geq3$. For example, If one wants to determine all curves $X_0(p)$ with infinitely many degree $9$ points, then one of the possibilities arising from \Cref{KVthm} would be the degree $3$ rational morphism from $X_0(p)$ to a curve with infinitely many cubic points.
\end{remark}

With \Cref{KVhighgenus} in hand, we only need to deal with primes smaller than $200$. Note that in \Cref{inf_deg6points_section} we have already proved that most curves $X_0(p)$ with $p<200$ actually have infinitely many degree $6$ points. 
So, in order to prove \Cref{deg6_thm} it remains to show that $X_0(p)$ has only finitely many points for $p \in \set{ 157, 173, 197, 199}$.

Let $C$ be a curve over $\Q$ with at least one rational point, and let  $\Sym^d(C)$ be the $d$-th symmetric power of $C$ on $\Q$. We define \[\phi_d:\Sym^d(C)\to\Pic^d (C)\] as the morphism that sends an effective divisor of degree $d$ to its associated line bundle. Furthermore, we use \begin{align}W_d^0 C := \textup{Im } \phi_d \subset \Pic^d(C)\label{eq:Wd0}\end{align} to denote the scheme theoretic image of $\phi_d$. Note that $W_d^0(C)$ can also be defined as \[W_d^r(C):=\{D\in\Pic^d(C): \ell(D)>r\}.\] These two definitions are equivalent; see \cite[Sections 2 and 4.2]{DHJOdensitydegree} for more details.

One way to characterize the infinitude of degree $d$ points on $C$ is the following theorem.

\begin{theorem}[{\cite[Theorem 4.2. (1)]{belov}}]\label{translate_of_abelian_variety_thm}
    Let $C$ be a curve over a number field $K$. Then $C$ has infinitely many degree $d$ points over $K$ if and only if at least one of the following two statements holds:
    \begin{enumerate}
        \item There exists a map $C\to \mathbb{P}^1$ of degree $d$ over $K$.
        \item There exists a degree $d$ point $x\in C$ and a positive rank abelian subvariety $A\subset \Pic^0 (C)$ such that $\phi_d(x_1+\ldots+x_d)+A\subset W^0_dC$ where $x_1,\ldots,x_d$ are the Galois conjugates of $x$.
    \end{enumerate} 
\end{theorem}

\begin{proposition}[{\cite[Proposition 3.6]{DebarreFahlaoui}}]\label{DF_AV_dim}
    Let $C$ be a curve over $\C$ of genus $g$. Assume that $W_d^r(C)$ contains an abelian variety $A$ and that $d\leq g-1+r$. Then $\dim A\leq d/2-r$.
\end{proposition}

\begin{corollary}\label{AV_cor}
    The modular curve $X_0(p)$ has only finitely many degree $6$ points over $\Q$ for $p \in \set{157,173, 199}$
\end{corollary}
\begin{proof}
    In all these cases, we have $g(X_0(p))\geq7$ and the $\Q$-gonality of $X_0(p)$ is at least $7$ by \cite[Theorem 1.1]{NajmanOrlic22}. \Cref{translate_of_abelian_variety_thm} now tells us that if the curve $X_0(p)$ has infinitely many degree $6$ points, then $W_6^0(C)$ must contain a translate of a positive rank abelian variety $A$. By \Cref{DF_AV_dim}, we must have $\dim A\leq3$.

    However, in all these cases, all simple abelian subvarieties $A$ of dimension $\leq3$ in the decomposition of the Jacobian $J_0(p)$ have analytic rank $0$. This can be checked on LMFDB by searching all newforms of level $p$, weight $2$, dimension $1-3$, and character order $1$. \begin{center}
    \url{https://www.lmfdb.org/ModularForm/GL2/Q/holomorphic/?level=157%2C173%2C199&weight=2&char_order=1&showcol=analytic_rank}.
    \end{center} Therefore, their rank is also $0$ by Kolyvagin's theorem (generalized by Kato \cite[Corollary 14.3]{kato2004p}) and we get a contradiction.
\end{proof}

The only remaining levels are $p=193,197$. These two curves $X_0(p)$ are of genera $15$ and $16$, respectively. In addition, there exist positive rank abelian varieties $A\subset J_0(p)$ of dimension $\leq3$ and we cannot use \Cref{AV_cor} to eliminate them.

We first deal with $p=197$. The only positive rank simple abelian variety $A\subset J_0(197)$ is the elliptic curve $E$ with LMFDB label \lmfdbec{197}{a}{1}. It has modular degree equal to $10$. This can be checked on \begin{center}
\url{https://www.lmfdb.org/ModularForm/GL2/Q/holomorphic/?level=197&weight=2&char_order=1}.
\end{center}  

\begin{proposition}\label{prop197}
The variety $W_6^0 X_0(197)_\Q$ does not contain a translate of a positive rank abelian variety. Consequently, the modular curve $X_0(197)$ has only finitely many degree $6$ points over $\Q$.
\end{proposition}
\begin{proof}
According to the LMFDB, $J_0(197)$ has only two simple isogeny factors of positive rank. These are the elliptic curve corresponding to the newform \lmfdbmf{197}{2}{a}{a} and a $5$-dimensional abelian variety corresponding to the newform \lmfdbmf{197}{2}{a}{b}. A translate of the $5$-dimensional variety cannot be contained in $W_6^0 X_0(197)$ due to \Cref{DF_AV_dim}. Note that the newform only specifies an elliptic curve up to isogeny. However, the elliptic curve with the LMFDB label \lmfdbec{197}{a}{1} corresponding to the modular form \lmfdbmf{197}{2}{a}{a} is the only one in its isogeny class.

Let $E$ be the elliptic curve \lmfdbec{197}{a}{1}. To show that $W_6^0 X_0(197)$ does not contain a translate of $E$, we use the same argument as in the proof of \cite[Proposition 4.9]{DHJOdensitydegree}. We explicitly computed the modular parametrization $X_0(197) \to E$. Using \texttt{Magma} computations over $\F_3$ we showed that $W_6^0 X_0(197)_{\F_3}$ does not contain a translate of $E_{\F_3}$. This was done by enumerating the finitely many points in $W_6^0 X_0(197)(\F_3)$ and for each such point $x$ we found a point $P \in E(\F_3)$ such that $x+P \notin W_6^0 X_0(197)(\F_3)$. After that, we use \cite[Lemma 4.8]{DHJOdensitydegree} to conclude that $W_6^0 X_0(197)_\Q$ does not contain a translate of $E$.

Since the $\Q$-gonality of $X_0(197)$ is equal to $8$ by \cite[Theorem 1.1]{NajmanOrlic22}, \Cref{translate_of_abelian_variety_thm} now tells us that $X_0(197)$ has only finitely many degree $6$ points.
\end{proof}

Combining the results of \Cref{inf_deg6points_section} and \Cref{finite_deg6points_section} gives us the proof of \Cref{deg6_thm}.

\begin{proof}[Proof of \Cref{deg6_thm}]
    For all primes $p$ listed in \Cref{deg6_thm} we prove that the curve $X_0(p)$ has infinitely many degree $6$ points in \Cref{deg6map} and \Cref{bars_cubic_points}. For all other primes $p$ we prove that $X_0(p)$ has only finitely many degree $6$ points in \Cref{KVhighgenus}, \Cref{DF_AV_dim}, and \Cref{prop197}.
\end{proof}

For the end of this paper we give a short discussion about the remaining level $p=193$. Unfortunately, so far we could not determine whether it has infinitely many degree $6$ points.

The curve $X_0(193)$ has genus $g=15$ and $\Q$-gonality equal to $8$ by \cite[Corollary 3.24]{Orlic2023}. We also know that it has only finitely many points of degree $\leq5$ by \Cref{quarticthm} and \Cref{quinticthm}. By \Cref{translate_of_abelian_variety_thm} and \Cref{DF_AV_dim} there exists a positive rank abelian variety $A$ of dimension $\leq3$ whose translate is contained in $W_6(X_0(193))$. There is only such possible variety $A$ and it is of dimension $2$.

Debarre-Fahlaoui curves play an important role in results classifying curves with infinitely many degree points, as can be seen in \cite{DebarreFahlaoui, KadetsVogt}, for example. We will not give a definition of the Debarre-Fahlaoui curves here. Instead, we present one necessary condition for the curve $C/\Q$ to be Debarre-Fahlaoui. The reason for this is that we do not know how to prove that a curve is Debarre-Fahlaoui. The definition can be found in \cite[Definition 5.1]{KadetsVogt}.

\begin{proposition}[{\cite[Definition 4.4, Proposition 4.6]{DHJOdensitydegree}}]\label{DF_proposition}
    If $C/\Q$ is the normalization of a Debarre-Fahlaoui curve of genus $\geq1$, then $C$ admits a non-constant morphism to a positive rank elliptic curve.
\end{proposition}

Since there are no elliptic curves in the decomposition of $J_0(193)$ (as there are no elliptic curves over $\Q$ with conductor $193$), \Cref{DF_proposition} implies that $X_0(193)$ is not Debarre-Fahlaoui.

Suppose that $\min(\delta(X_0(193)/\Q))=6$. \Cref{thm_p<3000} tells us that the quotient map $X_0(193)\to X_0^+(193)$ is the only non-constant rational morphism from $X_0(193)$ to a curve of genus $\geq2$. As the curve $X_0^+(193)$ has only finitely many cubic points by \cite[Theorem 1]{BarsDalal22}, we conclude that $X_0(193)$ is $6$-minimal.

These results give us a strong indication that the curve $X_0(193)$ has only finitely many degree $6$ points, see \cite[Theorems 5.8, 6.2]{KadetsVogt}.

\bibliographystyle{siam}
\bibliography{references}

@misc{lmfdb,
  shorthand    = {LMFDB},
  author       = {The {LMFDB Collaboration}},
  title        = {The {L}-functions and Modular Forms Database},
  howpublished = {\url{http://www.lmfdb.org}},
  year         = {2021},
  note         = {[Online; accessed 5 February 2021]},
}

@article {magma,
    AUTHOR = {Bosma, Wieb and Cannon, John and Playoust, Catherine},
     TITLE = {The {M}agma algebra system. {I}. {T}he user language},
      NOTE = {Computational algebra and number theory (London, 1993)},
   JOURNAL = {J. Symbolic Comput.},
  FJOURNAL = {Journal of Symbolic Computation},
    VOLUME = {24},
      YEAR = {1997},
    NUMBER = {3-4},
     PAGES = {235--265},
      ISSN = {0747-7171},
   MRCLASS = {68Q40},
       DOI = {10.1006/jsco.1996.0125},
       URL = {http://dx.doi.org/10.1006/jsco.1996.0125},
}

@article{kato2004p,
  title={$p$-adic {H}odge theory and values of zeta functions of modular forms},
  author={Kato, Kazuya},
  journal={Ast{\'e}risque},
  volume={295},
  pages={117--290},
  year={2004},
  publisher={[Paris: Societe mathematique de France, 1973-}
}

@article{kenkumomose1988automorphism,
  title={{Automorphism groups of the modular curve $X_0(N)$}},
  author={Kenku, Monsur A and Momose, Fumiyuki},
  journal={Compos. Math},
  fjournal={Compositio Mathematica},
  volume={65},
  number={1},
  pages={51--80},
  year={1988},
  publisher={Cambridge University Press}
}

@Article{belov,
 Author = {Bourdon, Abbey and Ejder, {\"O}zlem and Liu, Yuan and Odumodu, Frances and Viray, Bianca},
 Title = {On the level of modular curves that give rise to isolated {{\(j\)}}-invariants},
 FJournal = {Advances in Mathematics},
 Journal = {Adv. Math.},
 ISSN = {0001-8708},
 Volume = {357},
 Pages = {33},
 Note = {Id/No 106824},
 Year = {2019},
 Language = {English},
 DOI = {10.1016/j.aim.2019.106824},
 zbMATH = {7130219},
 Zbl = {1472.11190}
}

@Misc{DHJOdensitydegree,
 Author = {Derickx, Maarten and Hwang, Wontae and Jeon, Daeyeol and Orli{\'c}, Petar},
 Title = {{Modular curves $X_0(N)$ of density degree $5$}},
 Year = {2025},
 Eprint = {arXiv:2503.08975},
 Note = {preprint. available at: \url{https://arxiv.org/abs/2503.08975}},
}

@Article{KadetsVogt,
 Author = {Kadets, Borys and Vogt, Isabel},
 Title = {Subspace configurations and low degree points on curves},
 FJournal = {Advances in Mathematics},
 Journal = {Adv. Math.},
 ISSN = {0001-8708},
 Volume = {460},
 Pages = {36},
 Note = {Id/No 110021},
 Year = {2025},
 Language = {English},
 DOI = {10.1016/j.aim.2024.110021},
 zbMATH = {7966462}
}

@book{Shimura:AutomorphicForms,
  author		= {Goro Shimura},
  title		= {Introduction to the theory of automorphic forms},
  publisher 	= {Princeton University Press},
  year		= {1971}
}

@book {cox,
    AUTHOR = {Cox, David A.},
     TITLE = {Primes of the form {$x^2 + ny^2$}},
    SERIES = {A Wiley-Interscience Publication},
      NOTE = {Fermat, class field theory and complex multiplication},
 PUBLISHER = {John Wiley \& Sons Inc.},
   ADDRESS = {New York},
      YEAR = {1989},
     PAGES = {xiv+351},
      ISBN = {0-471-50654-0; 0-471-19079-9},
   MRCLASS = {11A41 (11F11 11R11 11R16 11R18 11R37 11Y11)},
MRREVIEWER = {Andrew Bremner},
}

@preamble{
   "\def\cprime{$'$} "
}

@book {modular,
    AUTHOR = {Diamond, Fred and Shurman, Jerry},
     TITLE = {A first course in modular forms},
    SERIES = {Graduate Texts in Mathematics},
    VOLUME = {228},
 PUBLISHER = {Springer-Verlag, New York},
      YEAR = {2005},
     PAGES = {xvi+436},
      ISBN = {0-387-23229-X},
   MRCLASS = {11Fxx},
MRREVIEWER = {Henri Darmon},
}

@article{GonzalezJimenez2002,
  title = {Modular curves of genus 2},
  volume = {72},
  ISSN = {0025-5718},
  url = {http://dx.doi.org/10.1090/S0025-5718-02-01458-8},
  DOI = {10.1090/s0025-5718-02-01458-8},
  number = {241},
  journal = {Math. Comp.},
  fjournal = {Mathematics of Computation},
  publisher = {American Mathematical Society (AMS)},
  author = {González-Jiménez,  Enrique and González,  Josep},
  year = {2002},
  month = jun,
  pages = {397–419}
}

@article{BakerGonzalezPoonen2005,
 ISSN = {00029327, 10806377},
 author = {Matthew H. Baker and Enrique González-Jiménez and Josep González and Bjorn Poonen},
 journal = {Amer. J. Math.},
 fjournal = {American Journal of Mathematics},
 number = {6},
 pages = {1325--1387},
 publisher = {Johns Hopkins University Press},
 title = {Finiteness Results for Modular Curves of Genus at Least 2},
 urldate = {2025-05-15},
 volume = {127},
 year = {2005}
}

@article {frey,
    AUTHOR = {Frey, Gerhard},
     TITLE = {Curves with infinitely many points of fixed degree},
   JOURNAL = {Israel J. Math.},
  FJOURNAL = {Israel Journal of Mathematics},
    VOLUME = {85},
      YEAR = {1994},
    NUMBER = {1-3},
     PAGES = {79--83},
      ISSN = {0021-2172},
   MRCLASS = {11G30 (11G05 14G25 14H25)},
MRREVIEWER = {Takeshi Ooe},
       DOI = {10.1007/BF02758637},
       URL = {https://doi.org/10.1007/BF02758637},
}

@article {abramovich,
    AUTHOR = {Abramovich, Dan},
     TITLE = {A linear lower bound on the gonality of modular curves},
   JOURNAL = {Internat. Math. Res. Notices},
  FJOURNAL = {International Mathematics Research Notices},
      YEAR = {1996},
    NUMBER = {20},
     PAGES = {1005--1011},
      ISSN = {1073-7928},
   MRCLASS = {11G18 (11F32 14G35)},
MRREVIEWER = {M. Ram Murty},
       DOI = {10.1155/S1073792896000621},
       URL = {https://doi.org/10.1155/S1073792896000621},
}

@article {DebarreFahlaoui,
    AUTHOR = {Debarre, Olivier and Fahlaoui, Rachid},
     TITLE = {Abelian varieties in {$W^r_d(C)$} and points of bounded degree
              on algebraic curves},
   JOURNAL = {Compositio Math.},
  FJOURNAL = {Compositio Mathematica},
    VOLUME = {88},
      YEAR = {1993},
    NUMBER = {3},
     PAGES = {235--249},
      ISSN = {0010-437X},
   MRCLASS = {14H30 (11G30 14H25 14H40)},
  MRNUMBER = {1241949},
MRREVIEWER = {John B. Little},
       URL = {http://www.numdam.org/item?id=CM_1993__88_3_235_0},
}

@article {Bars99,
    AUTHOR = {Bars, Francesc},
     TITLE = {Bielliptic modular curves},
   JOURNAL = {J. Number Theory},
  FJOURNAL = {Journal of Number Theory},
    VOLUME = {76},
      YEAR = {1999},
    NUMBER = {1},
     PAGES = {154--165},
      ISSN = {0022-314X},
   MRCLASS = {11G18 (11G30)},
  MRNUMBER = {1688168},
MRREVIEWER = {Joseph H. Silverman},
       DOI = {10.1006/jnth.1998.2343},
       URL = {https://doi.org/10.1006/jnth.1998.2343},
}

@article {Ogg74,
    AUTHOR = {Ogg, Andrew P.},
     TITLE = {Hyperelliptic modular curves},
   JOURNAL = {Bull. Soc. Math. France},
  FJOURNAL = {Bulletin de la Soci\'{e}t\'{e} Math\'{e}matique de France},
    VOLUME = {102},
      YEAR = {1974},
     PAGES = {449--462},
      ISSN = {0037-9484},
   MRCLASS = {14G05 (10D05 14H45)},
  MRNUMBER = {364259},
MRREVIEWER = {Kuang-yen Shih},
       URL = {http://www.numdam.org/item?id=BSMF_1974__102__449_0},
}

@incollection {stein07,
    AUTHOR = {William Arthur Stein},
     TITLE = {Modular forms, a computational approach},
   JOURNAL = {Graduate Studies in Mathematics},
    VOLUME = {79},
     PUBLISHER = {Amer. Math. Soc., Providence, RI},
     YEAR = {2007},
}

@article{HarrisSilverman91,
 ISSN = {00029939, 10886826},
 author = {Joe Harris and Joe H. Silverman},
 journal = {Proc. Amer. Math. Soc.},
 fjournal = {Proceedings of the American Mathematical Society},
 number = {2},
 pages = {347--356},
 publisher = {American Mathematical Society},
 title = {Bielliptic Curves and Symmetric Products},
 volume = {112},
 year = {1991}
}

@article{FurumotoHasegawa1999,
author = {Furumoto, Masahiro and Hasegawa, Yuji},
title = {{Hyperelliptic Quotients of Modular Curves $X_0(N)$}},
volume = {22},
journal = {Tokyo J. Math.},
fjournal = {Tokyo Journal of Mathematics},
number = {1},
publisher = {Publication Committee for the Tokyo Journal of Mathematics},
pages = {105 -- 125},
year = {1999},
doi = {10.3836/tjm/1270041616},
URL = {https://doi.org/10.3836/tjm/1270041616}
}

@Article{HasegawaShimura1999,
 Author = {Hasegawa, Yuji and Shimura, Mahoro},
 Title = {Trigonal modular curves {{\(X_0^{+d}(N)\)}}},
 FJournal = {Proceedings of the Japan Academy. Series A},
 Journal = {Proc. Japan Acad., Ser. A},
 ISSN = {0386-2194},
 Volume = {75},
 Number = {9},
 Pages = {172--175},
 Year = {1999},
 Language = {English},
 DOI = {10.3792/pjaa.75.172},
 zbMATH = {1443393},
 Zbl = {0963.11031}
}

@article{JEON2018319,
title = {{Bielliptic modular curves $X_0^+(N)$}},
journal = {J. Number Theory},
fjournal = {Journal of Number Theory},
volume = {185},
pages = {319-338},
year = {2018},
issn = {0022-314X},
doi = {https://doi.org/10.1016/j.jnt.2017.09.006},
url = {https://www.sciencedirect.com/science/article/pii/S0022314X17303475},
author = {Daeyeol Jeon},
}

@Article{Jeon2021,
 Author = {Jeon, Daeyeol},
 Title = {Modular curves with infinitely many cubic points},
 FJournal = {Journal of Number Theory},
 Journal = {J. Number Theory},
 ISSN = {0022-314X},
 Volume = {219},
 Pages = {344--355},
 Year = {2021},
 Language = {English},
 DOI = {10.1016/j.jnt.2020.09.006},
 zbMATH = {7276963},
 Zbl = {1469.11196}
}

@article{Hwang2023,
  doi = {10.1090/mcom/3864},
  url = {https://doi.org/10.1090/mcom/3864},
  year = {2023},
  month = aug,
  publisher = {American Mathematical Society ({AMS})},
  author = {WonTae Hwang and Daeyeol Jeon},
  title = {Modular curves with infinitely many quartic points},
  journal = {Math. Comp.},
  fjournal = {Mathematics of Computation},
}

@Article{Poonen2007,
 Author = {Poonen, Bjorn},
 Title = {Gonality of modular curves in characteristic {{\(p\)}}},
 FJournal = {Mathematical Research Letters},
 Journal = {Math. Res. Lett.},
 ISSN = {1073-2780},
 Volume = {14},
 Number = {4},
 Pages = {691--701},
 Year = {2007},
 Language = {English},
 DOI = {10.4310/MRL.2007.v14.n4.a14},
 zbMATH = {5240894},
 Zbl = {1138.14016}
}

@unpublished {NguyenSaito,
AUTHOR = {Khac Viet Nguyen and Masa-Hiko Saito},
    TITLE = {$d$-gonality of modular curves and bounding torsions},
   YEAR = {1996},
   NOTE = {preprint, available at: \url{https://arxiv.org/abs/alg-geom/9603024}}
}

@Article{JeonPark05,
 Author = {Jeon, Daeyeol and Park, Euisung},
 Title = {Tetragonal modular curves},
 FJournal = {Acta Arithmetica},
 Journal = {Acta Arith.},
 ISSN = {0065-1036},
 Volume = {120},
 Number = {3},
 Pages = {307--312},
 Year = {2005},
 Language = {English},
 DOI = {10.4064/aa120-3-6},
 zbMATH = {5002948},
 Zbl = {1165.11322}
}

@Article{HasegawaShimura_trig,
 Author = {Hasegawa, Yuji and Shimura, Mahoro},
 Title = {Trigonal modular curves},
 FJournal = {Acta Arithmetica},
 Journal = {Acta Arith.},
 ISSN = {0065-1036},
 Volume = {88},
 Number = {2},
 Pages = {129--140},
 Year = {1999},
 Language = {English},
 DOI = {10.4064/aa-88-2-129-140},
 zbMATH = {1310234},
 Zbl = {0947.11018}
}

@article{BarsDalal22,
AUTHOR = {Francesc Bars and Tarun Dalal},
    TITLE = {Infinitely many cubic points for {$X_0^+(N)$} over {$\Q$} },
    JOURNAL = {Acta. Arith.},
    FJOURNAL = {Acta Arithmetica},
    VOLUME = {206},
    PAGES = {373--388},
    YEAR = {2022},
}

@article{NajmanOrlic22,
  doi = {10.1090/mcom/3873},
  url = {https://doi.org/10.1090/mcom/3873},
  year = {2023},
  month = jul,
  publisher = {American Mathematical Society ({AMS})},
  author = {Filip Najman and Petar Orli{\'{c}}},
  title = {{Gonality of the modular curve $X_0(N)$}},
  journal = {Math. Comp},
  fjournal = {Mathematics of Computation}
}

@article{DerickxOrlic23,
  doi = {10.1007/s40993-024-00525-6},
  year = {2024},
  volume = {10},
  number = {42},
  author = {M. Derickx and P. Orli{\'{c}}},
  title = {{Modular curves $X_0(N)$ with infinitely many quartic points}},
  journal = {Res. Number Theory},
  fjournal = {Research in Number Theory},
}

@article{Orlic2023,
Author = {Petar Orlić},
Title = {{Tetragonal modular quotients $X_0^+(N)$}},
Year = {2025},
Journal = {Acta Arith.},
FJournal = {Acta Arithmetica},
Volume = {217},
Pages = {261--275}
}

@book{serre2016topics,
 Author = {Serre, Jean-Pierre},
 Title = {Topics in {Galois} theory. {Notes} written by {Henri} {Darmon}},
 Edition = {2nd},
 FSeries = {Research Notes in Mathematics},
 Series = {Res. Notes Math.},
 Volume = {1},
 ISBN = {978-1-56881-412-4},
 Year = {2007},
 Publisher = {Wellesley, MA: A K Peters},
 Language = {English},
 zbMATH = {5219694},
 Zbl = {1128.12001}
}

@article{Kim2002,
  title = {{Functoriality for the exterior square of $\textup{GL}_4$ and the symmetric fourth of $\textup{GL}_2$. Appendix 2: Refined estimates towards the Ramanujan and Selberg conjectures (by Henry Kim and Peter Sarnak)}},
  volume = {16},
  ISSN = {1088-6834},
  url = {http://dx.doi.org/10.1090/S0894-0347-02-00410-1},
  DOI = {10.1090/s0894-0347-02-00410-1},
  number = {1},
  FJournal = {Journal of the American Mathematical Society},
  Journal = {J. Amer. Math. Soc.},
  publisher = {American Mathematical Society (AMS)},
  author = {Kim,  Henry},
  year = {2002},
  month = oct,
  pages = {139–183}
}

@article{DerickxNajman2023,
    author = {M. Derickx and F. Najman},
    title = {{Hyperelliptic and trigonal modular curves in characteristic $p$}},
    fjournal = {The Quarterly Journal of Mathematics},
    journal = {Q. J. Math.},
    volume = {00},
    pages = {1--18},
    year = {2024},
    issn = {0033-5606},
    doi = {10.1093/qmath/haae059},
    eprint = {https://academic.oup.com/qjmath/advance-article-pdf/doi/10.1093/qmath/haae059/60916051/haae059.pdf},
}

@article{Ramar2001,
  title = {{Approximate formulae for $L(1, \chi)$}},
  volume = {100},
  ISSN = {1730-6264},
  url = {http://dx.doi.org/10.4064/aa100-3-2},
  DOI = {10.4064/aa100-3-2},
  number = {3},
  journal = {Acta Arith.},
  fjournal = {Acta Arithmetica},
  publisher = {Institute of Mathematics,  Polish Academy of Sciences},
  author = {Ramaré,  Olivier},
  year = {2001},
  pages = {245–266}
}

@article{Cowan2022,
  title = {Computing newforms using supersingular isogeny graphs},
  volume = {8},
  ISSN = {2363-9555},
  url = {http://dx.doi.org/10.1007/s40993-022-00392-z},
  DOI = {10.1007/s40993-022-00392-z},
  number = {4},
  journal = {Res. Number Theory},
  fjournal = {Research in Number Theory},
  publisher = {Springer Science and Business Media LLC},
  author = {Cowan,  Alex},
  year = {2022},
}

@inbook{Best2021,
  title = {Computing Classical Modular Forms},
  ISBN = {9783030809140},
  ISSN = {2365-9572},
  url = {http://dx.doi.org/10.1007/978-3-030-80914-0_4},
  DOI = {10.1007/978-3-030-80914-0_4},
  booktitle = {Arithmetic Geometry,  Number Theory,  and Computation},
  publisher = {Springer International Publishing},
  author = {Best,  Alex J. and Bober,  Jonathan and Booker,  Andrew R. and Costa,  Edgar and Cremona,  John E. and Derickx,  Maarten and Lee,  Min and Lowry-Duda,  David and Roe,  David and Sutherland,  Andrew V. and Voight,  John},
  year = {2021},
  pages = {131–213}
}
\end{document}